%% file: Index_Theorem_article.tex
\documentclass[a4paper, 13pt]{article}
\input{Packages}

\input{EngcommandsForArticle}

\input{PersonalCommands}

\AtEveryBibitem{\clearfield{url}}

\newcommand{\gr}{\mathfrak{gr}}
\newcommand{\Gr}{\mathcal{G}\mathrm{r}}
\newcommand{\aF}{\mathfrak{a}\cF}

\newcommand{\taD}{\mathfrak{t}\aF}

\begin{document}
\title{On the index of	maximally hypoelliptic differential operators}
\author{Omar Mohsen}
\date{}
\maketitle

% 19K35 Kasparov theory ($KK$-theory)
% 19K56 Index theory
% 22A22 Topological groupoids (including differentiable and Lie groupoids)
% 22A25 Representations of general topological groups and semigroups
% 22E25 Nilpotent and solvable Lie groups
% 46L87 Noncommutative differential geometry
% 46L45 Decomposition theory for $C^*$-algebras
% 46L80 $K$-theory and operator algebras (including cyclic theory)
% 46L87 Noncommutative differential geometry
% 47G30 Pseudodifferential operators	
% 53C12 Foliations (differential geometric aspects)
% 53C29 Issues of holonomy
% 58J22 Exotic index theories
% 58J40 Pseudodifferential and Fourier integral operators on manifolds
% 53R30 Foliations; geometric theory
% 58B34 Noncommutative geometry (�  la Connes)
% 58H05 Pseudogroups and differentiable groupoids 
% 93B18 Linearizations
\begin{abstract} 
	We give an index formula for the class of all $*$-maximally hypoelliptic differential operators on any closed manifold with vector bundle coefficients, generalising previous index formulas by Atiyah-Singer and van Erp. Using this formula, we give new explicit index computations for Hörmander's sum of squares operators of arbitrary rank.\blfootnote{AMS subject classification: 19K56, 22A22, 46L80,~ Secondary: 22E25, 53R30. Keywords: index theory, maximally hypoelliptic differential operators, singular foliations, $C^*$-algebras, Rockland condition.}
	\end{abstract}
%Our formula is stated and proved more generally for longitudinally maximally hypoelliptic differential operators on foliated manifolds.
\setcounter{tocdepth}{2} %doesn't display subsections in TOC 
\tableofcontents
	\section*{Introduction}
Let $E,F\to M$ be vector bundles on a closed manifold $M$, $D:\Gamma(E)\to  \Gamma(F)$ a differential operator. Under suitable conditions on $D$, one can insure that $\ker(D)$ and $\ker(D^*)$ are finite dimensional, where $D^*$ is the formal adjoint. One then defines the analytic index by \begin{equation}\label{eqn:analytic index cinf}
 \Ind_a(D):=\dim(\ker(D))-\dim(\ker(D^*)).
\end{equation}
The index problem asks for a topological formula for the analytic index. For elliptic differential operators, such a formula was found by Atiyah and Singer \cite{AtiyahSingerI}. Nevertheless there is a very large class of non-elliptic operators for which the analytic index is finite.
 
 For some non-elliptic operators called Toeplitz operators which are defined on the boundary of a complex domain, a topological formula for the analytic index was found by Boutet de Monvel \cite{BoutetIndex}. Based on work by Epstein and Melrose \cite{MelroseEps2}, the formula of Boutet de Monvel was generalised to arbitrary contact manifolds by van Erp for operators without vector bundle coefficients \cite{Erik1,Erik2}. Later van Erp and Baum generalised van Erp's formula to allow vector bundle coefficients \cite{ErpBaum} (always on contact manifolds). Their formula allows them to compute the index of some Hörmander's sum of squares  \cite{Hormander:SoS} of rank $2$. 

 In this article, we consider the class of $*$-maximally hypoelliptic differential operator. A $*$-maximally hypoelliptic differential operator is a differential operator $D$ on a manifold $M$ (not necessarily contact) such that one can find some Sobolev type Hilbert spaces $H_1$ and $H_2$ such that $D$ admits a bounded Fredholm extension $D:H_1\to H_2$, see Theorem \ref{thm:max hypo} and Corollary \ref{thm:max hypo_cor} for a precise definition, see also the AMS notice by Street \cite{AMSHypoSurvey} for an expository introduction to the topic. In this case, one can show that the analytic index is finite and it coincides with the Fredholm index of $D:H_1\to H_2$, see Section \ref{sec:setting}. Elliptic differential operators are $*$-maximally hypoelliptic because one can take $H_1$ and $H_2$ the classical Sobolev spaces. Hörmander's sum of squares are also $*$-maximally hypoelliptic see \cite{RotschildStein,FollandStein,HelfferNourrigatBook}. The operators considered by Boutet de Monvel and more generally by van Erp and Baum are also $*$-maximally hypoelliptic, see \cite{BeaGre}. We refer the reader to Example \ref{ex:examplehjisdqf}, for concrete examples of $*$-maximally hypoelliptic differential operators on $S^1\times S^1$ which aren't elliptic.

Atiyah-Singer topological formula is based on the classical principal symbol. In \cite{MohsenMaxHypo}, we defined a principal symbol for $*$-maximally hypoelliptic differential operators. Theorem \ref{thm:max hypo} which is the main theorem of \cite{MohsenMaxHypo}, previously conjectured by Helffer and Nourrigat \cite{HelfferNourrigatCRAcSci} and proven in some cases by them in \cite{HelfferNourrigatBook}, gives an equivalence between $*$-maximally hypoellipticity and invertibility of our principal symbol

Theorem \ref{main thm}, the main theorem of this article is a topological index formula using our symbol from \cite{MohsenMaxHypo} for any $*$-maximally hypoelliptic differential operators on any closed manifold with vector bundle coefficients. This generalises the previously mentioned index theorems and allows us to give new computations of the analytic index. We show computability of our theorem by giving new explicit examples of index computations of Hörmander's sum of squares operators of arbitrary rank on $S^1\times S^1$ and more generally of some $*$-maximally hypoelliptic operators on any manifold $M$ fibered over $S^1$. The computation is a very straightforward application of our index formula.

 Like the approaches of van Erp \cite{Erik1,Erik2} and van Erp and Baum \cite{ErpBaum}, our approach is based on Connes's proof of the Atiyah-Singer index theorem \cite{ConnesBook}. The main novelty in this article is the use of the $C^*$-algebra of singular foliations introduced by Androulidakis and Skandalis \cite{AS1} and their blowups introduced by the author \cite{MohsenBlowup}. In the proof of our index formula, we need a localized version of the Connes-Thom isomorphism \cite{ConnesThom} which is proved in Section \ref{sec:variant}.
\paragraph{Structure of the paper.} \begin{itemize}
\item In Section \ref{sec:setting}, we give a general overview of the article. We briefly recall the results of \cite{MohsenMaxHypo}. We define maximal hypoelliptic operators. We then state the main theorem of \cite{MohsenMaxHypo}, which gives a simple criteria equivalent to maximal hypoellipticity. We then state our index formula. In the end we give examples of index computations of Hörmander's sum of squares operators on $S^1\times S^1$.
\item In \cite{ConnesThom}, Connes proved the Connes-Thom isomorphism, which implies that $K_*(C^*G)\simeq K^*(\mathfrak{g}^*)$, where $G$ is any simply connected nilpotent Lie group. In Section \ref{sec:variant}, we prove a localized version of Connes's isomorphism, where we show that $K_*(I_Z)\simeq K^*(Z)$, where $Z\subset \mathfrak{g}^*$ is any locally closed subset which is invariant under the co-adjoint action and dilations, and $I_Z$ is the corresponding subquotient of $C^*G$.
\item In Section \ref{sec:Examp}, we use our index formula to compute the index of some $*$-maximally hypoelliptic differential operators on any manifold $M$ fibered over $S^1$.
\item In Section \ref{sec:index theorem}, we prove the index formula.\end{itemize}
\section*{Acknowledgments}	
Part of this work was done while the author was a postdoc in Muenster university and was funded by the Deutsche Forschungsgemeinschaft (DFG, German Research Foundation) - Project-ID 427320536 - SFB 1442, as well as Germany's Excellence Strategy EXC 2044 390685587, Mathematics M\"{u}nster: Dynamics-Geometry-Structure. 

The author thanks P. Pansu for interesting discussions which ultimately lead to the examples in the article. The author also thanks G. Skandalis and S. Vassout for their valuable remarks.
\section{Maximal hypoellipticity and index formula}\label{sec:setting}
\paragraph{Setting.} Let $M$ be a closed manifold, $N\in \N$ called the depth and $$0\subseteq \cF^1\subseteq \cdots\subseteq \cF^N=\cX(M)$$ finitely generated $C^\infty(M)$-modules of vector fields which satisfy the condition \begin{equation}\label{eqn:BracketF}
 [\cF^i,\cF^j]\subseteq \cF^{i+j},\quad  i+j\leq N
\end{equation} For example if $X_1,\cdots,X_r\in \cX(M)$ satisfy Hörmander's condition of rank $N$, then one can define $\cF^i$ inductively by \begin{align*}
\cF^1=\langle X_1,\cdots,X_r\rangle,\quad \cF^k=\cF^{k-1}+[\cF^{1},\cF^{k-1}].
\end{align*} 
In general, we don't suppose that the modules $\cF^i$ are generated by the iterated Lie brackets of $\cF^1$.

 Let $x\in M$. We define the Lie algebra $$\gr(\cF)_x=\oplus_{i=1}^N\frac{\cF^i}{\cF^{i-1}+I_x\cF^i},$$ where $I_x\subseteq C^\infty(M,\R)$ is the ideal of smooth functions vanishing at $x$. The Lie algebra structure on $\gr(\cF)_x$ comes from the Lie bracket of vector fields by the formula \begin{align*}[\cdot,\cdot]:\frac{\cF^i}{\cF^{i-1}+I_x\cF^i}\times \frac{\cF^j}{\cF^{j-1}+I_x\cF^j}&\to \frac{\cF^{i+j}}{\cF^{i+j-1}+I_x\cF^{i+j}}\\
   [[X]_x,[Y]_x]&\mapsto\big[[X,Y]_x\big],\quad X\in \cF^i,Y\in \cF^j.
\end{align*} The Lie algebra $\gr(\cF)_x$ is a graded nilpotent Lie algebra. We remark that the function $x\mapsto \dim(\gr(\cF)_x)$ isn't locally constant in general. We denote by $\Gr(\cF)_x$ the space $\gr(\cF)_x$ seen as a Lie group by Baker–Campbell–Hausdorff formula. More generally throughout the article if $\mathfrak{g}$ is a nilpotent Lie algebra, then the simply connected Lie group integrating $\mathfrak{g}$ denotes the space $\mathfrak{g}$ with the product by Baker–Campbell–Hausdorff formula.

 \paragraph{Principal symbol.} 
 
 Let $E,F\to M$ be vector bundles, $D:\Gamma(E)\to \Gamma(F)$ a differential operator on $M$. We say that $D$ is of order $\leq k$ if $D$ can be written as sum of monomials of the form \begin{equation}\label{eqn:diff k}
 \nabla_{X_1}\cdots \nabla_{X_s}\phi,
\end{equation} where $\phi\in \Gamma(\End(E,F))$, $\nabla$ a connection on $F$, $X_i\in \cF^{a_i}$ and $\sum_{i=1}^sa_i\leq k$. We denote by $\Diff^k_{\cF}(M,E,F)$ the space of differential operators of order $\leq k$.

 Let $\pi:\Gr(\cF)_x\to U(H)$ be an irreducible unitary representation. If $D\in \Diff^k_{\cF}(M,E,F)$ is written as sum of monomials of the form $\nabla_{X_1}\cdots \nabla_{X_s}\phi $ as in \eqref{eqn:diff k}, then one defines \begin{equation}\label{eqn:princip symb dfn}
\textstyle \sigma^k(D,x,\pi)=  \sum^\prime \big( \pi([X_1]_x)\cdots \pi([X_s]_x)\big)\otimes \phi_x:C^\infty(H)\otimes E_x\subseteq H\otimes E_x\to  H\otimes F_x,
\end{equation} where \begin{itemize}
 \item $\Sigma'$ means one sums only over monomials such that $\sum_{i=1}^s a_i=k$.
 \item $[X_i]_x\in \frac{\cF^{a_i}}{\cF^{a_i-1}+I_x\cF^{a_i}}\subseteq \gr(\cF)_x$ is the class of $X_i$.
 \item $\pi([X_i]_x)$ is the differential of $\pi$ applied to $[X_i]_x$.
 \item $C^\infty(H)$ denotes the subspace of smooth vectors.
\end{itemize}  
The sum in \eqref{eqn:princip symb dfn} can depend on the choice of the presentation of $D$ as a sum of monomials, see \cite[Examples 1.14 and the following paragraph]{MohsenMaxHypo}. A non trivial fact is that it doesn't depend on the choice of the presentation for representations in the characteristic set which is defined as follows. Let \begin{align}\label{eqn:af*}
 \gr(\cF)^*:=\sqcup_{x\in M}\gr(\cF)^*_x,\quad  \aF^*:=T^*M\times ]0,1]\sqcup \gr(\cF)^*\times \{0\}.
\end{align}The space $\aF^*$ is equipped with the weakest topology such that \begin{itemize}
\item the natural projection map $\aF^*\to M\times [0,1]$ is continuous
 \item for every $i\in \{1,\cdots,N\}$, $X\in \cF^i$, the map \begin{align*}
\aF^*&\to \R\\
 (\xi,x,t)&\to \xi(X(x))t^i,\quad t\in ]0,1],x\in M,\xi\in T_x^*M\\
 (\xi,x,0)&\to \xi ( [X]_x),\quad x\in M,\xi \in \gr(\cF)_x^*
 \end{align*} is continuous, where $[X]_x\in\frac{\cF^i}{\cF^{i-1}+I_x\cF^i}\subseteq \gr(\mathcal{D})_{x}$ denotes the class of $X$.
\end{itemize}
The space $\aF^*$ is locally compact Hausdorff by \cite[Prop. 2.10]{AS2}, see \cite[Section 1.3]{MohsenMaxHypo}. The set $T^*M\times ]0,1]$ isn't dense in $\aF^*$ in general. The topological characteristic set $\cT^*\cF\subseteq \gr(\cF)^*$ is defined by $$\cT^*\cF=\{\xi \in\gr(\cF)^*:(\xi,0)\in \overline{T^*M\times ]0,1]}\}.$$ We also define $\cT^*\cF_x:=\cT^*\cF\cap  \gr(\cF)_x^*$. 
\begin{theorem}[{\cite[Theorem B]{MohsenMaxHypo}}]\label{thm:charset}For any $x\in M$, \begin{itemize}
\item the set $\cT^*\cF_x$ is closed under the co-adjoint action of $\Gr(\cF)_x^*$. \item for any $\xi\in\cT^*\cF_x$, $\sigma^k(D,x,\pi_\xi)$ is well defined (independent of the choice of a presentation $D$ as sum of monomials), where $\pi_\xi$ is an irreducible unitary representation which corresponds to $\xi$ by the orbit method \cite{KirillovArticle,BrownArticleTopOrbitMethod}.
\end{itemize}
\end{theorem}
\paragraph{Sobolev spaces.}
We define the Sobolev space $H^k_{\cF}(M,E)$ to be the set of distributions $u$ on $M$ with coefficients in $E$ such that for any differential operator\footnote{In \cite{MohsenMaxHypo}, we defined a pseudo-differential calculus where differential operators of order $k$ are pseudo-differential of order $k$. Generally one needs to allow $D$ to be a pseudo-differential operator. If $\cF^i$ are generated by iterated Lie brackets of $\cF^1$, then one can restrict to $D$ being a differential operator.} $D\in \Diff^k_{\cF}(M,E,E)$, one has $Du\in L^2(M,E)$. In \cite{MohsenMaxHypo}, we extend the definition to $H^s_{\cF}(M,E)$ for any $s\in \R$. It satisfies \begin{equation}\label{eqn:int_Sobolev_spaces}
 \bigcap_{s\in \R}H^s_{\cF}(M,E)=\Gamma(E)
\end{equation}
and if $s\in \R$ and $D\in \Diff^k_{\cF}(M,E,F)$, then $D:H^s_{\cF}(M,E)\to H^{s-k}_{\cF}(M,F)$ is bounded.
\paragraph{Maximal hypoellipticity.}
 The following theorem was conjectured by Helffer and Nourrigat in \cite{HelfferNourrigatCRAcSci}, see \cite[Conjecture 2.3]{HelfferNourrigatBook}. The implication ($2\implies 1$) was proven in \cite{HelfferNourrigatBook} as well as $(1\implies 2)$ in some cases.
\begin{theorem}[{\cite[Theorem C]{MohsenMaxHypo}}]\label{thm:max hypo}Let $D\in \Diff_{\cF}^k(M,E,F)$. The following are equivalent \begin{enumerate}
\item for any $x\in M$ and $\xi\in \cT^*\cF_x\backslash\{0\}$, the map $\sigma^k(D,x,\pi_\xi)$ is injective.
\item for any $s\in \R$, $u$ a distribution on $M$ with coefficients in $E$. If $Du\in H^{s-k}_{\cF}(M,F)$, then $u\in H^s_{\cF}(M,E)$.
\item there exists $s\in \R$ such that for any $u$ a distribution on $M$ with coefficients in $E$. If $Du\in H^{s-k}_{\cF}(M,F)$, then $u\in H^s_{\cF}(M,E)$.
\item for any $s\in \R$, the operator $D:H^s_{\cF}(M,E)\to H^{s-k}_{\cF}(M,F)$ is left invertible modulo compact operators.
\item there exists $s\in \R$, such that the operator $D:H^s_{\cF}(M,E)\to H^{s-k}_{\cF}(M,F)$ is left invertible modulo compact operators.
%\item for any $s>0$, there exists $C>0$ such that for any $u\in H^{k+s}_{\cF}(M,E)$, $$\norm{u}_{H^{k+s}_{\cF}(M,E)}\leq C(\norm{Du}_{H^{k}_{\cF}(M,F)}+\norm{u}_{H^{k}_{\cF}(M,E)}).$$
\end{enumerate}
 We say that $D$ is maximally hypoelliptic if it satisfies the above conditions. We say that $D$ is $*$-maximally hypoelliptic if $D$ and its formal adjoint $D^*$ are maximally hypoelliptic.
\end{theorem}
Atkinson's theorem implies the following \begin{cor}\label{thm:max hypo_cor}Let $D\in \Diff_{\cF}^k(M,E,F)$. The following are equivalent \begin{enumerate}
\item there exists $s\in \R$, such that the operator $D:H^s_{\cF}(M,E)\to H^{s-k}_{\cF}(M,F)$ is Fredholm
\item $D$ is $*$-maximally hypoelliptic
\end{enumerate}
\end{cor}
Elliptic operators are $*$-maximally hypoelliptic ($N=1$). Hörmander's sum of squares are also maximally hypoelliptic \cite{RotschildStein}. We refer the reader to \cite{HelfferNourrigatBook} for more details on maximally hypoelliptic operators.
\paragraph{Index formula.}Let $D\in \Diff^k_{\cF}(M,E,F)$ be a $*$-maximally hypoelliptic operator. By Theorem \ref{thm:max hypo} and Atkinson's theorem, $$D:H^s_{\cF}(M,E)\to H^{s-k}_{\cF}(M,F),\quad \forall s\in \R$$ is Fredholm. By Theorem \ref{thm:max hypo}.2 and \eqref{eqn:int_Sobolev_spaces}, $\ker(D)\subseteq \Gamma(E)$. Applying the same argument to $D^*$, we deduce that the Fredholm index of $D:H^s_{\cF}(M,E)\to H^{s-k}_{\cF}(M,F)$ doesn't depend on $s$ and is equal to the analytic index given by \eqref{eqn:analytic index cinf}. We now give a formula to compute the index using the principal symbol defined above.

 Let $x\in M$. We define $$ C^*\cT\cF_x:= \frac{C^*\Gr(\cF)_x}{\bigcap_{\xi\in \cT^*\cF_x}\ker(\pi_\xi)}.$$ In other words we only look at representations in the characteristic set. Even though the bundle $\Gr(\cF)_x$ as $x$ varies isn't a fiber bundle in the usual topological sense because of the jumps in dimension, one can still make sense of its $C^*$-algebra. In particular, in Section \ref{sec:variant}, we define a $C^*$-algebra $C^*\cT\cF$ which is fibered over $M$ with fiber $C^*\cT\cF_x$. We also define an unbounded multiplier $$\sigma^k(D):\dom(\sigma^k(D))\subseteq \Gamma(E)\otimes_{C(M)} C^*\cT\cF\to \Gamma(F)\otimes_{C(M)} C^*\cT\cF $$ which has the property that for $x\in M$ and $\xi\in \cT^*\cF_x$, one has \begin{equation}\label{eqn:princip rep eqn multi}
   \pi_\xi(\sigma^k(D))=\sigma^k(D,x,\pi_\xi).
\end{equation}

By \cite[Theorem 3.47]{MohsenMaxHypo}, $\sigma^k(D)$ is a Fredholm multiplier. This means that $\sigma^k(D)$ is regular in the sense of Baaj-Julg \cite{BaajJulg} and its bounded transform is Fredholm. Hence $$[\sigma^k(D)]\in K_0(C^*\cT\cF)$$ is well defined. To state our index formula, we need three maps.
\begin{itemize}
\item In Section \ref{sec:variant}, we define a map $$\mu_{\cT^*\cF}:K^0(\cT^*\cF)\to K_0(C^*\cT\cF)$$ which is a localized version of the Connes-Thom isomorphism. We show that $\mu_{\cT^*\cF}$ is an isomorphism. Here $K^0(\cT^*\cF)$ is the topological $K$-theory with compact support of the space $\cT^*\cF$ which inherits the subspace topology from $\aF^*$.
\item By applying Excision to the locally compact Hausdorff space $$T^*M\times ]0,1]\sqcup \cT^*\cF\times \{0\}\subseteq \aF^*,$$ we get a map $$\Ex:K^0(\cT^*\cF)\to K^0(T^*M).$$
\item We have the Atiyah-Singer index map $$\Ind_{AS}:K^0(T^*M)\to \Z.$$
\end{itemize}
\begin{thmx}\label{main thm}Let $D\in \Diff^k_{\cF}(M,E,F)$ be $*$-maximally hypoelliptic. Then $$\Ind_a(D)=\Ind_{AS}\circ \Ex\circ \mu_{\cT^*\cF}^{-1}([\sigma^k(D)]).$$ 
\end{thmx}
The maps $\Ex$ and $\Ind_{AS}$ are topological maps while $\mu_{\cT^*\cF}$ is given by excision at the level of noncommutative $C^*$-algebras. We prove that the map $\mu_{\cT^*\cF}$ is an isomorphism by using the series decomposition of the $C^*$-algebra of nilpotent Lie groups. We show that one has an isomorphism at every step in the decomposition. In principle, to invert $\mu$, one can go back through each isomorphism. This is essentially a Mayer Vietoris argument on the space of representations. This is rather straightforward to do if there aren't many representations of different type. One has the following theorem which gives restrictions on the characteristic set.
\begin{thmx}\label{second_main_thm}If $x\in M$, $\xi\in \cT^*\cF_x$, then the co-adjoint orbit $Ad^*(\Gr(\cF)_x)\cdot \xi$ has dimension $\leq 2\dim(M)$.
\end{thmx}
So even though $\dim(\Gr(\cF)_x)-\dim(M)$ is arbitrary large, the representations obtained from the characteristic set by the orbit method act on $L^2\R^d$ with $d\leq \dim(M)$. Theorem \ref{second_main_thm} is proved in the end of Section \ref{sec:index theorem}.

 Before we give an example, we remark that generally $\cT^*\cF_x$ is a proper subset of the following two sets and so neither of them can be used for our index formula \begin{itemize}
\item the set of $\xi\in \gr(\cF)_x^*$ such that $Ad^*(\Gr(\cF)_x)\cdot \xi$ has dimension $\leq 2\dim(M)$, see the examples in Section \ref{sec:Examp}.
\item the set of $\xi\in \gr(\cF)_x^*$ such that $\sigma^k(D,x,\pi_\xi)$ is well defined for any $k\in \N$, $D\in \Diff_{\cF}^k(M,E,F)$, see \cite[Examples 1.16]{MohsenMaxHypo}.
\end{itemize}
\begin{ex}\label{ex:examplehjisdqf}
 Let $f:S^1\to \R,$ $g:S^1\times S^1\to \C$ be smooth functions. We suppose that for each $x\in f^{-1}(0)$, the order of vanishing of $x$ is finite. Let $k$ be the maximum of the order of vanishing of $x$ for all $x\in f^{-1}(0)$ and $f^{-1}_k(0)$ the set of $x\in f^{-1}(0)$ with order of vanishing $k$. We suppose $k$ is odd. Consider the differential operator on $S^1\times S^1$  $$D=(\partial_x^2+(f(x)\partial_y)^2)^{\frac{k+1}{2}}+ g(x,y)\sqrt{-1}\partial_y$$
  Then $D$ is $*$-maximally hypoelliptic if and only if  $$g(x,y)\notin \pm\mathrm{spec}\left(\left(\partial_z^2-\left(\frac{f^{(k)}(x)}{k!}z^k\right)^{2}\right)^{\frac{k+1}{2}}\right),\quad \forall (x,y)\in f^{-1}_k(0)\times S^1,$$ where $\left(\partial_z^2-\left(\frac{f^{(k)}(x)}{k!}z^k\right)^{2}\right)^{\frac{k+1}{2}}$ acts on $L^2(\R)$. One then has \begin{equation}\label{eqn:index sec 1}
     \Ind_a(D)=\sum_{x\in f^{-1}_k(0) }\sum_{\lambda}w(\lambda-g(x,\cdot))-w(\lambda + g(x,\cdot)).
\end{equation} where $w$ is the winding number around $0$, and $\sum_{\lambda}$ is the sum over the spectrum of $$\left(\partial_z^2-\left(\frac{f^{(k)}(x)}{k!}z^k\right)^{2}\right)^{\frac{k+1}{2}}.$$ One can suppose that $g\in \Gamma(\End(E))$ where $E$ is a vector bundle, in which case one replaces $\partial_x$ and $\partial_y$ with $\nabla_{\partial_x},\nabla_{\partial_y}$ for any connection on $E$ and the winding number with the $K^1$ class $[\lambda\pm g(x,\cdot)]\in K^1(S^1)=\Z$. This example is generalised and worked out in details in Section \ref{sec:Examp}. To this end we need the definition of the map $\mu_{\cT^*\cF}$ in the next section.
\end{ex}
\section{Localized Connes-Thom isomorphism}\label{sec:variant}
\paragraph{Excision map.}Let $B$ be a $C^*$-algebra fibered over $[0,1]$ in the sense of Kasparov \cite{KasparovInvent}, $A$ and $C$ the fibers at $1$ and $0$ respectively. We suppose that $B$ is constant over $]0,1]$. Hence, we have an exact sequence \begin{equation}\label{eqn:index map seq}
 0\to A\otimes C_0(]0,1])\to B\xrightarrow{\ev_0} C\to 0.
\end{equation} Since $C_0(]0,1])$ is contractible, the $6$-term exact sequence implies that $$K(\ev_0):K_*( B)\to K_*(C)$$ is an isomorphism. We define \begin{equation}
\mu:K_*(C)\xrightarrow{K(\ev_0)^{-1}}K_*(B)\xrightarrow{K(\ev_1)}K_*(A).
\end{equation} 
 This construction is functorial, in the following sense. Suppose one has a $C([0,1])$-homomorphism map $\phi:B\to B'$. Then the following diagram commutes $$\begin{tikzcd}K_*(C)\arrow[r,"\mu"]\arrow[d]&K_*(A)\arrow[d,"\phi"]\\K_*(C')\arrow[r,"\mu'"]&K_*(A').
 \end{tikzcd}$$
\paragraph{Subquotients of $C^*$-algebras.} Let $A$ be a $C^*$-algebra, $Z$ a locally closed subsets of the spectrum $\hat{A}$ of $A$. One defines a subquotient of $A$ as follows. If $Z=D\cap F\subseteq \hat{A}$ with $D$ open and $F$ closed, then one defines \begin{align*}
 I_Z:=\frac{\bigcap_{\pi \in D^c}\ker(\pi)}{\bigcap_{\pi \in D^c\cup F}\ker(\pi)}.
\end{align*}
The $C^*$-algebra $I_Z$ up to a canonical isomorphism is independent of the choice of $D,F$ such that $Z=D\cap F$. Moreover $\hat{I}_Z$ is homeomorphic to $Z$. This gives an bijection between locally closed subsets of $\hat{A}$ and subquotients of $A$, see \cite[Remark 2.3]{FourierTransformNilpotentGroups} and \cite[Lemma 7.5.3]{PhillipsEquivKtheoryBook}.
\paragraph{Localized Connes-Thom.} Let $\mathfrak{g}$ be a nilpotent Lie algebra, $G$ the simply connected Lie group integrating $\mathfrak{g}$. We define a family of Lie algebras $(\mathfrak{g}_t)_{t\in [0,1]}$ where $\mathfrak{g}_t=\mathfrak{g}$ with the Lie bracket $$[X,Y]_{\mathfrak{g}_t}=t[X,Y].$$ Let $G_t$ be the simply connected Lie group integrating $\mathfrak{g}_t$, $\overline{G}=\sqcup_{t\in [0,1]}G_t$ the bundle over $[0,1]$ of groups. Then the groups $G_t$ for $t>0$ are isomorphic to $G$ by multiplication by $t$. Hence one has short exact sequence $$0\to C^*G\otimes C_0(]0,1])\to C^*\overline{G}\to C_0(\mathfrak{g}^*)\to 0.$$
By \cite[Chapter 3 Theorem 1]{LudwigBook}, the spectrum $\widehat{C^*\overline{G}}=\sqcup_{t\in ]0,1]}\mathfrak{g}^*/Ad^*(G)\sqcup \mathfrak{g}^*$ and the topology is the quotient of $\mathfrak{g}^*\times [0,1]$. In the previous identification, if $Z\subset\mathfrak{g}^*\times \{t\}$ is a closed under the coadjoint action, then it corresponds to a subset $\widehat{C^*G_t}$ which then corresponds to a subset of $\widehat{C^*G}$ but there is a factor $t$ in this identification. This is irrelevant in what follows because we will only take $\R_+^\times$-invariant sets. 

Let $Z\subseteq \mathfrak{g}^*$ be a locally closed subset which is $\R_+^\times$ and $Ad^*(G)$ invariant. The $\R_+^\times$-structure on $\mathfrak{g}$ comes from the standard vector space dilation. The set $Z\times [0,1]\subseteq \mathfrak{g}^*\times [0,1]$ gives a locally closed subset of $\widehat{C^*\overline{G}}$. Let $\bar{I}_Z:=I_{Z\times [0,1]}$ be the corresponding subquotient. Then one has a short exact sequence $$0\to I_Z\otimes C_0(]0,1])\to \bar{I}_Z\to C_0(Z)\to 0.$$ Hence we get a map $$\mu_Z:K^*(Z)\to K_*(I_Z),$$ where $K^*(Z)$ is topological $K$-theory with compact support of $Z$.
\begin{theorem}\label{thm 1}
The map $\mu_Z$ is an isomorphism.
\end{theorem}
\begin{proof}
We fix a decomposition $Z=D\cap F$ with $D$ open and $F$ closed throughout the proof. %It is enough to prove that $K_*(\ker(\ev_1))=0$.
We will use the decomposition of $C^*G$ given in \cite[Theorem 4.11]{FourierTransformNilpotentGroups} based on the stratification in \cite{PedersenGeometricQuantNilpotent} of the unitary dual. There exists for some $n\in \N$, a sequence $$0=\cJ_0\subsetneq \cJ_1\subsetneq\cdots \subsetneq\cJ_n=C^*G$$ of ideals which satisfy the properties in \cite[Definition 2.9]{FourierTransformNilpotentGroups}. We will use the notation used there. Each ideal $\cJ_{i}=I_{\tilde{V_i}}$ where $\tilde{V_i}\subseteq \mathfrak{g}^*$ is an $\R_+^\times$ and $Ad^*(G)$ invariant open set. The quotient $V_i:=\tilde{V}_i/G\subseteq \mathfrak{g}^*/G$ is described in \cite[Proposition 4.5]{FourierTransformNilpotentGroups}. Since each $\tilde{V_i}$ is $\R_+^\times$-invariant and $Ad^*(G)$-invariant, one has a decomposition  $$0=\bar{I}_{Z\cap \tilde{V}_0}\subset \bar{I}_{Z\cap \tilde{V}_1}\subset\cdots \subset \bar{I}_{Z\cap \tilde{V}_n}=\bar{I}_Z.$$ For each $i$, there is a exact sequences $$0\to I_{Z\cap \tilde{V}_i}\otimes  C_0(]0,1])\to \bar{I}_{Z\cap \tilde{V}_i}\to C_0(Z\cap \tilde{V}_{i})\to 0,$$ and hence maps $$\mu_i:K^*\left(Z\cap \tilde{V}_{i}\right)\to K_*\left(I_{Z\cap \tilde{V}_i}\right).$$
We will prove that each $i$ is an isomorphism by induction. To this end,  the short exact sequence $$0\to I_{Z\cap (\tilde{V}_{i}\backslash \tilde{V}_{i-1})}\otimes C_0(]0,1]) \to \bar{I}_{Z\cap (\tilde{V}_{i}\backslash \tilde{V}_{i-1})}\to C_0(Z\cap (\tilde{V}_{i}\backslash \tilde{V}_{i-1}))\to 0$$gives map \begin{equation}\label{eqn:mu proof}
 \mu_{i}':K^*\left(Z\cap (\tilde{V}_{i}\backslash \tilde{V}_{i-1})\right)\to K_*\left(I_{Z\cap (\tilde{V}_{i}\backslash \tilde{V}_{i-1})}\right).
\end{equation} Functoriality of $\mu$, the $6$-term exact sequence in $K$-theory and the $5$-lemma imply that it is enough to prove that for each $i$, $\mu_i'$ is an isomorphism. It is shown in \cite[Lemma 1.6.1]{PedersenGeometricQuantNilpotent} that $\tilde{V}_{i}\backslash \tilde{V}_{i-1}$ is diffeomorphic to $(V_i\backslash V_{i-1})\times \R^{2d}$ for some $d$ (which may depend on $i$). Since $Z$ is $Ad^*(G)$ invariant, it is mapped by this diffeomorphism to $\left((Z/Ad^*(G))\cap (V_i\backslash V_{i-1})\right)\times \R^{2d}$. Furthermore by \cite[Theorem 4.11]{FourierTransformNilpotentGroups} $V_i\backslash V_{i-1}$ is a semi-algebraic variety, in particular locally compact Hausdorff and $I_{\tilde{V}_{i}\backslash \tilde{V}_{i-1}}\simeq C_0(V_i\backslash V_{i-1})\otimes \cK(L^2\R^d)$. By Morita equivalence, we deduce that $$K_*\left(I_{Z\cap(\tilde{V_i}\backslash \tilde{V}_{i-1})}\right)=K^*((Z/Ad^*(G))\cap (V_i\backslash V_{i-1})).$$It follows that the two sides of \eqref{eqn:mu proof} are isomorphic by a Bott periodicity. It thus suffices to see that $\mu_i'$ is the Bott periodicity. This follows from \cite[Lemma 6 on Page 109]{ConnesBook}.\end{proof}
%To prove that $\mu_Z$ is an isomorphism, it suffices to show that $\ev_1:\bar{I}_Z\to I_Z$ is an isomorphism in $K$-theory. By the $6$-term exact sequence, it suffices to show that $K_*(\ker(\ev_1))=0$. We will show by induction that $\ker(\ev)_1\cap \bar{I}_{Z\cap \tilde{V}_i}=0 $ for all $i$. By the $6$-term exact sequence applied to the exact sequence $$0\to \ker(\ev_1)\cap \bar{I}_{Z\cap \tilde{V}_i}\to \ker(\ev_1)\cap \bar{I}_{Z\cap \tilde{V}_{i+1}}\to \ker(\ev_1)\cap \bar{I}_{Z\cap (\tilde{V}_{i+1}\backslash \tilde{V}_{i})}\to 0,$$ it is enough to show that $\ker(\ev_1)\cap \bar{I}_{Z\cap (\tilde{V}_{i+1}\backslash \tilde{V}_{i})}$ has vanishing $K$-theory. The latter $C^*$-algebra is the kernel of the evaluation at $1$-map $\bar{I}_{Z\cap (\tilde{V}_{i+1}\backslash \tilde{V}_{i})}\to I_{Z\cap (\tilde{V}_{i+1}\backslash \tilde{V}_{i})}$.

\paragraph{Generalisation.} Let $X$ be a locally compact Hausdorff space. Suppose that $Z\subseteq \mathfrak{g}^*\times X$ is a locally closed subset which is $\R_+^\times$-invariant and $Ad^*(G)$-invariant, where both act trivially on $X$. By identifying the spectrum of $C^*G\otimes C_0(X)$ with $(\mathfrak{g}^*/Ad^*(G))\times X$, we get a short exact sequence $$0\to I_Z\otimes C_0(]0,1])\to \bar{I}_Z\to C_0(Z)\to 0,$$ where $I_Z$ and $\bar{I}_Z$ are the subquotients of $C^*G\otimes C_0(X)$ and $C^*\overline{G}\otimes C_0(X)$ which correspond to $Z$ and $Z\times [0,1]$ respectively. It follows that we have a map $$\mu_Z:K^*(Z)\to K_*\left(I_Z\right)$$
\begin{theorem}\label{thm 2}The map $\mu_Z$ is an isomorphism.
\end{theorem}
The proof of Theorem \ref{thm 2} is the same as that of Theorem \ref{thm 1}.
\begin{rem}The construction of $\mu_Z$ can be improved to give an element of Kasparov's \cite{KasparovInvent} $\KK$-theory $\mu_Z\in \KK^0(C_0(Z),I_Z)$. Our proof shows that $\mu_Z$ is invertible.
\end{rem}
\paragraph{Application to the characteristic set.}Let $\cF^i$ be modules as in Section \ref{sec:setting}, $X_i^j$ a finite number of generators of $\cF^i$. Let $\mathfrak{g}$ be the free graded nilpotent Lie algebra of step $N$ with a generator $\tilde{X}_i^j$ of degree $i$ for each $i,j$, $G$ the simply connected Lie group integrating $\mathfrak{g}$. For each $x\in M$, one has a Lie group homeomorphism \begin{align}\label{eqn:G}
 G\xrightarrow{\phi_x} \Gr( \cF)_x,\quad \tilde{X}_i^j\to [X_i^j]_x\in \frac{\cF^i}{\cF^{i-1}+I_x\cF^i}.
\end{align}
Let $$\phi^*:\gr(\cF)^*\to \mathfrak{g}^*\times M, \quad \phi^*(\xi,x)=(\phi_x^*(\xi),x).$$ The map $\phi$ is a closed embedding when $\gr(\cF)^*$ is equipped with the subspace topology from $\aF^*$ defined in Section \ref{sec:setting}. The set $\phi^*(\cT^*\cF)$ is closed because $\cT^*\cF$ is closed. It is also $\R_+^\times$ invariant and $Ad^*(G)$-invariant by \cite[Proposition 1.11]{MohsenMaxHypo}. We define $$C^*\cT\cF=I_{\phi^*(\cT^*\cF)}$$ which is a quotient of $C^*G\otimes C(M)$. It is clear that $C^*\cT\cF$ is fibered over $M$ with fiber at $x\in M$ equal to $I_{\cT^*\cF_x}$. It also doesn't depend on the choice of generators $X_i^j$, because for any other set of generators one could take the union of the two sets and form the same construction. Finally by Theorem \ref{thm 2}, one has an isomorphism $$\mu_{\cT^*\cF}:K^*(\cT^*\cF)\to K_*(C^*\cT\cF).$$

\paragraph{Principal symbol as multiplier.}Recall that a left invariant differential operator on $G$ gives an unbounded multiplier of $C^*G$. Let $D\in \Diff^k_{\cF}(M,E,F)$. We write $D$ as a sum of monomials as in \eqref{eqn:diff k} but we only use vector fields $X_i^j$. So the monomials have the form $\nabla_{X_{a_1}^{b_1}}\cdots \nabla_{X_{a_s}^{b_s}}\phi$ with $\sum_{l=1}^s a_i\leq s$. For $x\in M$, we define an unbounded multiplier $ E_x\otimes C^*G\to F_x\otimes C^*G$ by taking the sum of $\tilde{X}_{a_1}^{b_1}\cdots \tilde{X}_{a_s}^{b_s}\otimes \phi_x$ over the monomials with $\sum_{l=1}^s a_i=s$. These multipliers as $x$ varies form a multiplier of $C^*G\otimes C(M)$. This multiplier depends on the choice of the presentation of $D$ as sum of monimals as well as the choice of the generators. Nevertheless any multiplier of $C^*G\otimes C(M)$ descends to any quotient and thus it descends to $I_{\phi^*(\cT^*\cF)}$, and gives a multiplier  \begin{equation}
\sigma^k(D):\dom(\sigma^k(D))\subseteq \Gamma(E)\otimes_{C(M)} C^*\cT\cF\to \Gamma(F)\otimes_{C(M)} C^*\cT\cF
\end{equation}
which satisfies \eqref{eqn:princip rep eqn multi}. By \eqref{eqn:princip rep eqn multi} and Theorem \ref{thm:charset}, the multiplier $\sigma^k(D)$ doesn't depend on the choice of the generators $X_i^j$ nor on the choice of the presentation of $D$ as sum of monomials.
%\begin{theorem}\label{thm:princip symbol Fredholm multi}
%If $D$ is $*$-maximally hypoelliptic, then $\sigma^k(D)$ is a Fredholm multiplier. This means that it is regular in the sense of Baaj-Julg \cite{BaajJulg} and its bounded transform is Fredholm.
%\end{theorem}
%Theorem \ref{thm:princip symbol Fredholm multi} is a consequence \cite[Theorem 4.44]{MohsenMaxHypo}. 
\section{Examples of index computations}\label{sec:Examp}
The goal of this section is to compute the analytic index of some $*$-maximally hypoelliptic operators on closed manifolds fibered over $S^1$. To this end, we will first consider the model singularity for our examples. 
\paragraph{Model singularity.}Let $k,n\in \N$. Consider the filtration on $\R\times \R^n$ of depth $k+1$ given by $$\cF^i=\langle \partial_x,x^{k+1-i}\partial_{y_1},\cdots, x^{k+1-i}\partial_{y_n}\rangle.$$ The group $\Gr(\cF)_{(x,y)}$ has dimension $n+1$ with commutative structure if $x\neq 0$. If $x=0$, then $\gr(\cF)_{(0,y)}$ has dimension $1+n(k+1)$ with generators $[\partial_x]$ in degree $1$ and $[x^{i}\partial_{y_l}]$ in degree $k+1-i$ for $i\in \{0,\cdots,k\},l\in \{1,\cdots,n\}$ and the relations $[[\partial_x],[x^{i}\partial_{y_l}]]=i[x^{i-1}\partial_{y_l}]$. We denote by $\eta \in \gr(\cD)^*_{(0,y)}$ the dual of $[\partial_x]$ and by $\xi_{i,l}$ the dual of $[x^i\partial_y]$. By the definition of the characteristic set $\cT^*\cF_{(0,y)}$, it is the set of elements $(\eta,\xi)\in\gr(\cF)_{(0,y)}$ such that there exists a sequence $t_{\alpha}\in ]0,1],(x_\alpha,y_\alpha)\in \R\times \R^n$ and $(f_\alpha,g_\alpha)\in T^*_{(x_\alpha,y_\alpha)} (\R\times \R^n)$ such that $$(t_\alpha,x_\alpha,y_\alpha)\to (0,0,y),\quad t_\alpha f_\alpha\to \eta,\quad t_{\alpha}^{k+1-i}x_\alpha^{i}g_{\alpha,l}\to \xi_{i,l},\quad \forall \ i\in \{0,\cdots,k\},l\in \{1,\cdots,n\}$$ where $g_{\alpha,l}$ is the $l$-th component of $g_\alpha$. Hence it consists of elements $(\eta,\xi)$ such that $\eta\in \R$ is arbitrary and $\xi$ is either of the form $\xi_{i,l}=a^ib_l$ for some $a\in \R$ and $b\in \R^n$ or of the form $\xi_{i,l}=0$ if $i<k$ and $\xi_{k,l}=b_l$ for some $b\in \R^n$. So topologically the characteristic set is equal to $$\cT^*\cF_{(0,y)}\simeq \R\times \big(L\times \R_+/L\times \{0\}\big),$$ where $$L=S^{n-1}\times [-\infty,+\infty]/\sim$$ is the mapping torus where we identify $(x,+\infty)$ with $((-1)^kx,-\infty)$ for $x\in S^{n-1}$ and the quotient by $L\times \{0\}$ means we identify $L\times \{0\}$ with a point. 
%It follows that \begin{align*} K^0(\cT^*\cF_{(0,y)})=\frac{K^0(M_f)}{\Z}&=\begin{cases}\Z&\text{if}\ kn=0\mod 2 \quad \text{or}\ n=1\\0&\text{if not}\end{cases}\\K^1(\cT^*\cF_{(0,y)})=K^1(M_f)&=\begin{cases}\Z\oplus \Z&\text{if}\ kn=0\mod 2\quad \text{or}\ n=1\\\Z\oplus \frac{\Z}{2\Z }&\text{if not}\end{cases}\end{align*}

One checks that the set of representations in $\cT^*\cF_{(0,y)}$ obtained using the orbit method are \begin{itemize}
\item  the $1$-dimensional representations parameterised by $\R^{n+1}$
\item the representations on $L^2\R$ parameterised by $\delta\in \R^n\backslash \{0\}$ given by $$\pi_{\delta}([\partial_x])=\partial_x,\quad \pi_{\delta}([x^{i}\partial_{y_l}])=\sqrt{-1}\delta_l x^i.$$
\end{itemize}
Hence we have a short exact sequence $$0\to C_0(\R^n\backslash\{0\})\otimes \cK(L^2\R)\to C^*\cT\cF_{(0,y)}\to C_0(\R^{n+1})\to 0.$$This exact sequence corresponds by naturality of $\mu$ to the exact sequence $$0\to C_0(\R^3\times S^{n-1})\to C_0(\cT^*\cF_{(0,y)})\to C_0(\R^{n+1})\to 0.$$
\paragraph{Passage to manifolds.}Let $M$ be a closed $n+1$-dimensional manifold, $p:M\to S^1$ a fiber bundle, $f:S^1\to \R$ a smooth function which is not flat at any point in $f^{-1}(0)$. We denote by $k(x)$ the order of vanishing of $x\in f^{-1}(0)$, $k=\max_{x\in f^{-1}(0)}k(x)$, $f_k^{-1}(0)$ the set of $x\in f^{-1}(0)$ with $k(x)=k$. We fix $X\in \cX(M)$ a vector field such that $dp(X)=\partial_x$.

 We consider the filtration on $M$ of depth $k+1$ defined inductively by \begin{equation}\label{eqn:filtexampmanif}
   \cF^1=\langle X,(f\circ p)Y:Y\in \Gamma(\ker(dp))\rangle,\quad\cF^{i+1}=\cF^{i}+[\cF^i,\cF^1].
\end{equation}The module $\cF^1$ is locally free, so $\cF^1=\Gamma(E)$ for some vector bundle $E\to M$. The vector bundle $E$ is topologically isomorphic to $TM$. The isomorphism comes from the isomorphism \begin{align*}
 \cX(M)\to \cF^1=\Gamma(E),\quad gX+hY\mapsto gX+(f\circ p)hY,\quad g,h\in C^\infty(M),Y\in \ker(dp).
\end{align*} We equip $TM$ with a Riemannian metric such that the decomposition $TM=\R X\oplus \ker(dp)$ is orthogonal. We equip $T^*M$ with the dual metric. For $x\in f^{-1}(0)$, let $S_x$ be the sphere bundle of $T^*p^{-1}(x)$, $L_x$ the torus $S_x\times [-\infty,+\infty]/\sim$ with the identification $(y,+\infty)\sim ((-1)^{k(x)}y,-\infty)$. By the previous discussion, we can describe the characteristic set $\cT^*\cF$ as the gluing of $T^*M$ and $\bigsqcup_{x\in f^{-1}(0)}\R\times \big(L_x\times\R_+/L_x\times\{0\})$, where we identify for $x\in f^{-1}(0)$, $T^*M_{|p^{-1}(x)}=\R\oplus T^*p^{-1}(x)$ with $\R\times ((S_x\times \{0\}\times \R_+)/S_x\times\{0\})$. As a set $\cT^*\cF$ is equal to the disjoint union of $T^*M$ with $\R^3\times S_x$ for each $x\in f^{-1}(0)$. Since $T^*M$ is a closed subset of $\cT^*\cF$, we have an exact sequence \begin{align}\label{exactseq1}
0\to \bigoplus_{x\in f^{-1}(0)}C_0(\R^3\times S_x)\to C_0(\cT^*\cF)\to C_0(T^*M)\to 0
\end{align}

By the previous discussion, we also have a short exact sequence \begin{align}\label{exactseq2}
 0\to \bigoplus_{x\in f^{-1}(0)}C_0(\ring{T}^*p^{-1}(x)) \otimes \cK(L^2\R)\to C^*\cT\cF\to C_0(T^*M)\to 0,
\end{align} where $\ring{T}^*p^{-1}(x)$ is the cotangent bundle of the fiber $p^{-1}(x)$ with the zero section removed.

The short exact sequences \eqref{exactseq1} and \eqref{exactseq2} correspond to each other in $K$-theory using $\mu_{\cT^*\cF}$. One can create a similar exact sequence for $T^*M\times ]0,1]\sqcup \cT^*\cF\times\{0\}\subseteq \aF^*$ defined in \eqref{eqn:af*}, which corresponds to the above two sequence by $\Ex$ map. One then obtains the following. 
 \begin{prop}\label{prop:index comp} Let $D$ be a $*$-maximally hypoelliptic operator with respect to $\cF^\bullet$ of order $s\in \N$ whose principal symbol $[\sigma^s(D)]\in K_0(C^*\cT\cF)$ comes from an element in $$\alpha=\sum_{x\in f^{-1}(0)}\alpha_x\in K_0\left(\bigoplus_{x\in f^{-1}(0)}C_0(\ring{T}^*p^{-1}(x)) \otimes \cK(L^2\R)\right)=\bigoplus_{x\in f^{-1}(0)}K_0(C_0(\ring{T}^*p^{-1}(x)) \otimes \cK(L^2\R)).$$ Then $$\Ind_a(D)=\sum_{x \in f^{-1}(0)}\phi_x(\alpha_x),$$ where $\phi_x: K_0(C_0(\ring{T}^*p^{-1}(x)) \otimes \cK(L^2\R))\to\Z $ is the composition $$K_0(C_0(\ring{T}^*p^{-1}(x)) \otimes \cK(L^2\R))\xrightarrow{Morita}K^0(\ring{T}^*p^{-1}(x) )\xrightarrow{K_*(\iota)}K^0(T^*p^{-1}(x))\xrightarrow{Ind_{AS}}\Z,$$ where $\iota:\ring{T}^*p^{-1}(x)\to T^*p^{-1}(x)$ is the inclusion.
\end{prop}
\paragraph{Examples.}
Let $D,D':\Gamma(E)\to \Gamma(E)$ be differential operators of \textit{classical} order $1$ and $l\in \N$ respectively for some vector bundle $E\to M$ such that $2|l(k+1)$ and $D$ is longitudinal with respect to the foliation $\ker(dp)$ and $\sigma^1_{\text{classical}}(D,\xi)^2=-\norm{\xi}^2Id_E$ for $\xi\in \ker(dp)^*$. Here $\sigma^1_{classical}$ is the classical principal symbol. We fix any connection $\nabla$ on $E$. We consider the operator $$D''=(\nabla_X^2+((f\circ p) D)^2)^{\frac{l(k+1)}{2}}-D':\Gamma(E)\to \Gamma(E).$$
The operator $D''$ is of order $l(k+1)$ with respect to the filtration given by \eqref{eqn:filtexampmanif}. Let us compute its principal symbol. If $y\in M$, $\xi\in T^*_yM$ which gives a $1$-dimensional representation of $\Gr(\cF)_y$ by \eqref{exactseq1}, then \begin{equation}
 \sigma^{l(k+1)}(D'',y,\xi)=\sigma^{l(k+1)}((\nabla_X^2+((f\circ p) D)^2)^{\frac{l(k+1)}{2}},y,\xi)=\left(-\xi(X(y))^2-\norm{\xi_{|\ker(dp)}}^2\right)^{\frac{l(k+1)}{2}}.
\end{equation} The first equality follows because there is no term in $D'$ which is of order $l(k+1)$ and which is non zero when mapped by $\xi$. The injectivity $\sigma^{l(k+1)}(D'',y,\xi)$ if $\xi\neq 0$ is obvious. If $y\in p^{-1}(x)$ with $x\in f^{-1}(0)$, and $\xi\in \ring{T}^*p^{-1}(x)$ which corresponds to an infinite dimensional representation on $L^2\R$ by \eqref{exactseq1}, then by homogenity of the principal symbol with respect to graded dilations, it is enough to consider $\xi$ such that $\norm{\xi}=1$. There are two cases. It is convenient to denote by $c(x)=\frac{f^{(k(x))}(x)}{k!}$ for $x\in f^{-1}(0)$. \begin{itemize}
 \item Either $k(x)<k$. In this case none of the terms in $D'$ are of order $l(k+1)$ in a neighbourhood of $y$. Hence $$\sigma^{l(k+1)}(D'',y,\xi)=\left(\partial_z^2-c(x)^2z^{2k(x)}\right)^{\frac{l(k+1)}{2}}.$$ In this case injectivity is obvious, see \cite[Proposition 4.1]{MohsenMaxHypo}.
\item if $k(x)=k$, then the terms in $D'$ of order $l(k+1)$ are precisely the ones of order $l$ with respect to the classical order. It follows that $$\sigma^{l(k+1)}(D'',y,\xi)=\left(\partial_z^2-c(x)^2z^{2k}\right)^{\frac{l(k+1)}{2}}-\sigma^{l}_{classical}(D',y,\xi).$$
The operator $\left(\partial_z^2-c(x)^2z^{2k}\right)^{\frac{l(k+1)}{2}}$ acting on $L^2\R$ is diagonalizable. It follows that $D''$ is is $*$-maximally hypoelliptic if and only if 
$$\lambda-\sigma^l_{classical}(D',y,\xi)\in \GL(E_y),$$ for all $$\lambda\in \mathrm{spec}\left(\left(\partial_z^2-c(x)^2z^{2k}\right)^{\frac{l(k+1)}{2}}\right),\ y\in p^{-1}(f^{-1}_k(0)),\ \xi \in \ker(dp)_y,\ \norm{\xi}=1$$
\end{itemize}
By the computation above, $D''$ satisfies the hypothesis of Proposition \ref{prop:index comp}. So\begin{equation}\label{eqn:indexformulasec3Examp}
 \Ind_a(D'')=\sum_{x\in f^{-1}_k(0) }\sum_{\lambda} \Ind_{AS}([\lambda-\sigma^l_{classical}(D',y,\xi)]),
\end{equation} where \begin{itemize}
 \item $\sum_{\lambda} $ is the sum over the spectrum of $\left(\partial_z^2-c(x)^2z^{2k}\right)^{\frac{l(k+1)}{2}}$.
 \item $\Ind_{AS}([\lambda-\sigma^l_{classical}(D',y,\xi)])$ means the image of $[\lambda-\sigma^l_{classical}(D',y,\xi)]\in K^1(S_x)=K^0(\ring{T}^*p^{-1}(x))$ by $K^0(\ring{T}^*p^{-1}(x))\xrightarrow{K_*(\iota)}K^0(T^*p^{-1}(x))\xrightarrow{\Ind_{AS}}\Z$.
\end{itemize}
\begin{rem}\label{remendofsection}Consider the filtration of depth $k+2$ on $M$ \begin{equation*}
   \cF^1=\langle X\rangle,\cF^2=\langle X,(f\circ p)Y:Y\in \Gamma(\ker(dp))\rangle,\quad\cF^{i+1}=\cF^{i}+[\cF^i,\cF^1]+[\cF^{i-1},\cF^2].
\end{equation*}
The computations we did above work almost word for word the same for this filtration, one just replaces the formula for $D''$ with $$ (\nabla_X^4-((f\circ p) D)^2)^{\frac{l(k+2)}{4}}-D'$$ and $D,D'$ as above.
\end{rem}
\paragraph{Return to $S^1\times S^1$.} From \eqref{eqn:indexformulasec3Examp}, one immediately deduces \eqref{eqn:index sec 1}. We end this section two further examples. \begin{itemize}
\item for any $k\in \N$, consider the operator $$D=(\partial_x^2+(f(x)\partial_y)^2)^{k+1}+g(x,y)\partial_y^2.$$ The operator is $*$-maximally hypoelliptic if and only if $$g(x,y)\notin \mathrm{spec}\left(\left(\partial_z^2-c(x)^2z^{2k}\right)^{k+1}\right),\quad \forall (x,y)\in f^{-1}_k(0)\times S^1.$$ Its analytic index always vanishes by \eqref{eqn:indexformulasec3Examp}.
\item Using remark \ref{remendofsection}, supposing $4|k+2$, consider the operator $$D=(\partial_x^4-(f(x)\partial_y)^2)^{\frac{k+2}{4}}+g(x,y) \sqrt{-1}\partial_y.$$ The operator is $*$-maximally hypoelliptic if and only if $$g(x,y)\notin \pm\mathrm{spec}\left(\left(\partial_z^4+c(x)^2z^{2k}\right)^{\frac{k+2}{4}}\right),\quad \forall (x,y)\in f^{-1}_k(0)\times S^1.$$ One then has \begin{equation*}
     \Ind_a(D)=\sum_{x\in f^{-1}_k(0) }\sum_{\lambda}w(\lambda-g(x,\cdot))-w(\lambda + g(x,\cdot)),
\end{equation*} 
where the sum is over the spectrum.
\end{itemize}
\section{Proof of the index formula}\label{sec:index theorem}
In this section, we prove Theorems \ref{main thm} and \ref{second_main_thm}.
\begin{proof}[Proof of Theorem \ref{main thm}]
 The idea is to construct a $C^*$-algebra fibered over $[0,1]^2$ which connects the index for $*$-maximally hypoelliptic operators with Atiyah-Singer index. The $C^*$-algebra is obtained as the $C^*$-algebra of the blowup of a singular foliation on $M\times [0,1]^2$, see \cite{MohsenBlowup}. By a singular foliation on a compact manifold, we mean a finitely generated module of vector fields which is closed under Lie brackets. We define a singular foliation on $M\times [0,1]^2$ which is defined as the set of finite sums of vector fields of the form $$t^is f(x,t,s)X(x),$$ where $X\in \cF^i$ and $f\in C^\infty(M\times [0,1]^2)$. It is denoted by $\taD$. This is a singular foliation by \eqref{eqn:BracketF}. The holonomy groupoid $\cH(\taD)\rightrightarrows M\times [0,1]^2$ of $\taD$ defined in \cite{AS1} is fibered over $[0,1]^2$ with fiber over $(t,s)\in [0,1]^2$ equal to \begin{itemize}
\item the pair groupoid $M\times M$ if $(t,s)\in ]0,1]^2$, whose $C^*$-algebra is $\cK(L^2M)$
\item the bundle of abelian groups $TM$ if $(t,s)\in ]0,1]\times \{0\}$, whose $C^*$-algebra is $C_0(T^*M)$
\item the bundle of osculating groups $\Gr(\cF):=\sqcup_{x\in M}\Gr(\cF)_x$ if $(t,s)\in \{0\}\times ]0,1]$, whose $C^*$-algebra is $C^*\Gr(\cF)$ is equal to $I_{\phi^*(\gr(\cF)^*)}$, where $\phi$ is the map \eqref{eqn:G}. This follows from \cite[Proposition 2.18]{MohsenMaxHypo}.
\item the bundle of abelian groups $\bigsqcup_{x\in M}\gr(\cF)_x$ if $(t,s)=(0,0)$. By \cite[Theorem 3.7]{AS3}, the $C^*$-algebra of the fiber is equal to $C_0(\gr(\cF)^*)$, where $\gr(\cF)^*\subseteq \aF^*$ inherits the subspace topology defined in Section \ref{sec:setting}.
\end{itemize} 
We won't use the maximal $C^*$-algebra $C^*\taD$ defined in \cite{AS1} as it doesn't have the fibers we need. Instead we will use the $C^*$-algebra of the blowup $C^*_z\taD$ defined in \cite[Section 3]{MohsenBlowup}. By \cite[Theorem 3.7]{MohsenBlowup}, $C^*_z\taD$ can be defined as a quotient of $C^*\taD$ by the formula $$C^*_z\taD:=\frac{C^*\taD}{J},\quad J=\{a\in C^*\taD:a_{(t,s)}=0,\quad \forall (t,s)\in ]0,1]^2\}.$$
By \cite[Theorem 3.7 and Example 3.8]{MohsenBlowup}, we deduce that the fibers of $C^*_z\taD$ are given by\begin{itemize}
\item $\cK(L^2M)$  if $(t,s)\in ]0,1]^2$.
\item $C_0(T^*M)$ if $(t,s)\in ]0,1]\times \{0\}$.
\item $C^*\cT\cF$ if $(t,s)\in \{0\}\times ]0,1]$.
\item $C_0(\cT^*\cF)$ if $(t,s)=(0,0)$.
\end{itemize} 
The restriction to each side of the square $[0,1]^2$ gives a short exact sequence of the form \eqref{eqn:index map seq}, and hence an excision map $\mu$. The restriction to the side \begin{itemize}
\item $\{0\}\times [0,1]$ gives the short exact sequence $$0\to C^*\cT\cF\otimes C_0(]0,1])\to C^*_z\taD_{|\{0\}\times [0,1]}\to C_0(\cT^*\cF) \to 0$$ which gives the map $$\mu_{\cT^*\cF}:K^*(\cT^*\cF)\to K_*(C^*\cT\cF).$$
\item $\{1\}\times [0,1]$ gives the short exact sequence $$0\to \cK(L^2M)\otimes C_0(]0,1])\to C^*_z\taD_{|\{1\}\times [0,1]}\to C_0(T^*M) \to 0$$ which by \cite[Lemma 6 on Page 109]{ConnesBook} gives the Atiyah-Singer index map $$\Ind_{AS}:K^0(T^*M)\to \Z.$$
\item $[0,1]\times \{0\}$ gives the short exact sequence $$0\to C_0(T^*M)\otimes C_0(]0,1])\to C_0(T^*M\times ]0,1]\sqcup \cT^*\cF\times \{0\}) \to C_0(\cT^*\cF)\to 0$$ which gives the Excision map  $$\Ex:K^0(\cT^*\cF)\to K^0(T^*M).$$
\item $[0,1]\times \{1\}$ gives the short exact sequence $$0\to\cK(L^2M)\otimes C_0(]0,1])\to C^*_z\taD_{|[0,1]\times \{1\}}\to C^*\cT\cF\to 0$$ which gives a map $$\Ind_{\cF}:K_*(C^*\cT\cF)\to \Z.$$ \begin{lem}If $D\in \Diff^k_{\cF}(M,E,F)$ is $*$-maximally hypoelliptic, then $\Ind_{\cF}([\sigma^k(D)])=\Ind_a(D)$.
\end{lem}
\begin{proof}
In \cite[Section 2.4]{MohsenMaxHypo}, we constructed an unbounded multiplier $\Theta(D)$ of the $C^*$-algebra $C^*_z\taD_{|[0,1]\times \{1\}}$ which when restricted to $t\neq 0$ it acts on $\cK(L^2M)$ by $t^kD$ and when restricted to $t=0$ it acts on $C^*\cT\cF$ by $\sigma^k(D)$. This unbounded multiplier is Fredholm. This follows from the construction of a parametrix in \cite[Section 3.10 and 3.11]{MohsenMaxHypo}. It follows that $[\Theta(D)]\in K_0(C^*_z\taD_{|[0,1]\times \{1\})}$ is well defined. The definition of the map $\Ind_{\cF}$ implies the result, because $K(\ev_0)([\Theta(D)])=[\sigma^k(D)]$ and $K(\ev_1)([\Theta(D)])=\Ind_a(D)$.
\end{proof}
\end{itemize}
The existence of the $C^*$-algebra $C^*_z\taD$ on the square $[0,1]^2$ implies that\begin{equation}\label{eqn:comm A}
 \Ind_{\cF} \circ \mu_{\cT^*\cF}=\Ind_{AS}\circ \Ex.
\end{equation} To see this \begin{lem}The map $K(\ev_{(0,0)}):K_*(C^*_z\taD)\to K^*(\cT^*\cF)$ is an isomorphism, where $\ev_{(0,0)}$ is the evaluation map at $(0,0)$.
\end{lem}
\begin{proof}
The kernel of $\ev_{(0,0)}$ fits into the short exact sequences \begin{align*}
 0\to C^*_z\taD_{|]0,1]\times [0,1]}\to \ker(\ev_{(0,0)}) \to C^*_z\taD_{|\{0\}\times ]0,1]}\to 0\\
\end{align*}
Since $C^*_z\taD_{|\{0\}\times ]0,1]}\simeq C^*_z\taD_{|(0,1)}\otimes C_0(]0,1])$ and  $C^*_z\taD_{|]0,1]\times [0,1]}\simeq C^*_z\taD_{|\{1\}\times [0,1]}\otimes C_0(]0,1])$, the result follows from $6$-term exact sequence and the fact that $C_0(]0,1])$ is contractible.
\end{proof}
The two sides of \eqref{eqn:comm A} applied to an element of the form $K(\ev_{(0,0)})(x)$ is equal to $K(\ev_{(1,1)})(x)$ for $x\in K_*(C^*_z\taD)$. By surjectivity of $K(\ev_{(0,0)})$ the result follows. Since $\mu_{\cT^*\cF}$ is an isomorphism, we deduce that
\begin{equation}
 \Ind_{\cF}=\Ind_{AS}\circ \Ex\circ  \mu_{\cT^*\cF}^{-1}.
\end{equation}
This finishes the proof of Theorem \ref{main thm}.
\end{proof}
\begin{proof}[Proof of Theorem \ref{second_main_thm}]In \cite[Section 2.2]{MohsenBlowup}, a blowup space of any singular foliation is defined. In particular for $\taD$, we get a blowup space $\blup(\taD)_{(x,0,1)}$ which is defined as a subspace of the Grassmannian manifold of subspaces of $\gr(\cF)_x$ of codimension $\dim(M)$. It is shown in \cite[Section 3.2]{MohsenBlowup} that any element $\xi\in \cT^*\cF_x$ must vanish on a subspace $V\subseteq \gr(\cF)_x$ such that $V\in \blup(\taD)_{(x,0,1)}$. By \cite[Proposition 2.7]{MohsenBlowup}, $V$ is a Lie subalgebra of $\gr(\cF)_x$. If $B_\xi:\mathfrak{g}\times\mathfrak{g}\to \R$ is the anti-symmetric bilinear map $B_\xi(X,Y)=\xi([X,Y])$. Then it is well known that $\dim(Ad^*(\Gr(\cF)_x)\cdot \xi)=\codim(\ker(B_\xi))$. But since $\xi$ vanishes on $V$ and $V$ is a Lie subalgebra, it follows that $B_\xi$ vanishes on $V\times V$. The result follows.
\end{proof}
%One of the main properties of the $C^*$-algebra of singular foliations is that a vector field in the foliation acts by unbounded 
\begin{refcontext}[sorting=nyt]
\printbibliography
\end{refcontext}
{\footnotesize
		 (Omar Mohsen) Paris-Saclay University, Paris, France
		\vskip-2pt e-mail: \texttt{omar.mohsen@universite-paris-saclay.fr}}
\end{document}

%% file: Packages.tex
\usepackage[english]{babel}

\usepackage[T1]{fontenc}%allows one to use pipe sign, less than or equal, etc freely in tex file
\usepackage{lmodern}   
\usepackage{xcolor}%adds colors

\usepackage{amsmath,amsthm,amssymb,amsfonts}%AMS PACKAGE
\usepackage{mathtools}
\usepackage{tensor}%adds \tensor and \indices
\usepackage{tikz-cd}%adds diagrams
\usepackage{esint}%\fint
\usepackage{mathrsfs}%gives fancy mathscr
\usepackage{thmtools}
\usepackage{slashed}

\usepackage{bbm}
\usepackage{microtype}
\usepackage{accents}

\usepackage{enumitem}%enumerate package
\usepackage[toc,page]{appendix}%adds appendix
\usepackage{imakeidx}%allows one to maked indices in articles

\usepackage{crossreftools}
\usepackage{hyperref}
\usepackage{cleveref}

\usepackage{emptypage}%removes numbers from empty pages
\usepackage[a4paper,	bindingoffset=0in,	left=1.2in,	right=1.2in,	top=1in,	bottom=1in,	footskip=.25in]{geometry} %Margin
\usepackage[backend=biber,style=alphabetic,sorting=ynt,sortcites]{biblatex}%References

%% file: EngcommandsForArticle.tex
\allowdisplaybreaks[1]

\theoremstyle{plain}
\newtheorem{theorem}{Theorem}[section]
\newtheorem*{theorem*}{Theorem}
\newtheorem{thmx}{Theorem}

\newtheorem{cor}[theorem]{Corollary}
\newtheorem{lem}[theorem]{Lemma}
\newtheorem{prop}[theorem]{Proposition}
\theoremstyle{definition}
\newtheorem{ex}[theorem]{Example}

\newtheorem{rem}[theorem]{Remark}

\theoremstyle{remark}

\newcommand\blfootnote[1]{%
  \begingroup
  \renewcommand\thefootnote{}\footnote{#1}%
  \addtocounter{footnote}{-1}%
  \endgroup
}

%% file: PersonalCommands.tex
\newcommand{\RNum}[1]{\uppercase\expandafter{\romannumeral #1\relax}}

\makeatletter
\providecommand*{\twoheadrightarrowfill@}{%
  \arrowfill@\relbar\relbar\twoheadrightarrow
}
\providecommand*{\twoheadleftarrowfill@}{%
  \arrowfill@\twoheadleftarrow\relbar\relbar
}
\providecommand*{\xtwoheadrightarrow}[2][]{%
  \ext@arrow 0579\twoheadrightarrowfill@{#1}{#2}%
}
\providecommand*{\xtwoheadleftarrow}[2][]{%
  \ext@arrow 5097\twoheadleftarrowfill@{#1}{#2}%
}
\newcommand\setItemnumber[1]{\setcounter{enum\romannumeral\@enumdepth}{\numexpr#1-1\relax}}
\makeatother

\newcommand\norm[1]{\left\lVert#1\right\rVert}

\newcommand\KK[0]{K\! K}

\DeclareMathOperator\GL{GL}
\DeclareMathOperator\End{End}

\DeclareMathOperator\Ind{Ind}

\DeclareMathOperator\ev{ev}

\DeclareMathOperator\Diff{Diff}

\DeclareMathOperator\Ex{Ex}

\DeclareMathOperator\blup{blup}

\DeclareMathOperator\dom{Dom}
\DeclareMathOperator\codim{Codim}

\newcommand{\R}{\mathbb{R}}

\newcommand{\N}{\mathbb{N}}
\newcommand{\C}{\mathbb{C}}
\newcommand{\Z}{\mathbb{Z}}

\newcommand{\cD}{\mathcal{D}}

\newcommand{\cF}{\mathcal{F}}

\newcommand{\cH}{\mathcal{H}}

\newcommand{\cJ}{\mathcal{J}}
\newcommand{\cK}{\mathcal{K}}

\newcommand{\cT}{\mathcal{T}}

\newcommand{\cX}{\mathcal{X}}